\documentclass[12pt]{elsarticle}
\usepackage[right=2.cm,left=2.cm,top=2cm,bottom=2cm,headsep=0.5cm,footskip=0.5cm]{geometry}

\usepackage{graphicx,float,subfigure}
\usepackage{multicol,amsmath,amsfonts,amsthm,amssymb}
\usepackage{hyperref}
\hypersetup{urlcolor=blue, colorlinks=true, linkcolor=black,citecolor=black}
\usepackage{subfigure}

\newtheorem{theorem}{Theorem}

\theoremstyle{definition}
\newtheorem{example}{Example}
\usepackage{stmaryrd}

\usepackage{float} 

\makeatletter
\def\ps@pprintTitle{%
 \let\@oddhead\@empty
 \let\@evenhead\@empty
 \def\@oddfoot{\centerline{\thepage}}%
 \let\@evenfoot\@oddfoot}
\makeatother

\usepackage{fancyhdr}
\usepackage[symbol]{footmisc}

\begin{document}

\begin{frontmatter} 

\title{Global stability of a  distributed delayed viral model with general incidence rate}
\author{Eric \'Avila-Vales \footnote[1]{Corresponding author: avila@correo.uady.mx} , Abraham Canul-Pech and Erika Rivero-Esquivel \\
\small{\textit{Universidad Aut\'onoma de Yucat\'an, Facultad de Matem\'aticas}} \\
\small{\textit{Anillo Perif\'erico Norte, Tablaje 13615, C.P. 97119.
M\'erida, Yucat\'an, M\'exico.}}
}

\begin{abstract}
In this paper, we discussed an infinitely distributed delayed viral infection model with nonlinear immune response and general incidence rate. We proved the existence and uniqueness of the equilibria. By using the Lyapunov functional and LaSalle invariance principle, we obtained the conditions of global stabilities of
the infection-free equilibrium, the immune-exhausted equilibrium and the endemic equilibrium. Numerical simulations are given to verify the analytical results.
\end{abstract}

\begin{keyword}
CTL-response \sep Lyapunov functional\sep basic reproduction number \sep viral reproduction number 
\MSC{Primary: 34K20, secondary: 34D23}
\end{keyword}

\end{frontmatter}

\section{Introduction}

During recent decades there has been a lot of research regarding mathematical modelling of viruses dynamics via models of ordinary differential equations (ODE). Advances in immunology have lead us to better understand the interactions between populations of virus and the immune system, therefore several nonlinear sytems of ordinary differential equations (ODE's) has been proposed. Nowak and Bengham \cite{nowak1996population} study the model

\begin{align}
x'(t)&=s-dx(t)-\beta x(t) v(t),\\
y'(t)&= \beta x(t) v(t)-ay(t)-py(t)z(t),\\
v'(t)&= ky(t)-u v(t),\\
z'(t)&= cy(t)z(t)-bz(t).
\end{align}

Where $x(t)$ denotes the number of healthy cells, $y(t)$ denotes the infected cells, $v(t)$ denotes the number of matures viruses and $z(t)$  denotes the number of CTL (cytotoxic T lymphocyte response) cells. Uninfected target
cells are assumed to be generated at a constant rate
$s$ and die at rate $d$. Infection of target cells by free
virus is assumed to occur at rate $\beta$. Infected cells die at
rate $a$ and are removed at rate $p$ by the CTL immune
response. New virus is produced from infected cells at
rate $k$ and dies at rate $u$. The average lifetime of uninfected
cells, infected cells and free virus is thus given by
$1/d$, $1/a$ and $1/u$, respectively. $c$ denotes rate at which
the CTL response is produced, and $b$ denotes death rate
of the CTL response, respectively, with all given constants positive.

Generally, the type of incidence function used in a model has an important role in modeling the dynamics of viruses. The most common, is the bilinear incidence rate $\beta x v $. However, this rate is not useful all the time. For instance, using bilinear incidence suggests that model can not describe the infection process of hepatitis B, where individuals with small liver are more resistant to infection than the ones with a bigger liver. Recently, works about models of infections by viruses have used the incidence function of type Beddington-DeAngelis and Crowley-Martin. In \cite{xu2012global} the authors propose a model of infection by virus with Crowley-Martin functional response $\frac{\beta x(t)v(t)}{(1+ax(t))(1+bv(t))}$. Li and Fu in \cite{li2015global} study the following system:

\begin{align}
x'(t)&=s-dx(t)-\frac{\beta x(t)v(t)}{(1+ax(t))(1+b v(t))},\\
y'(t)&= \frac{\beta x(t-\tau)v(t)e^{-s\tau}}{(1+ax(t-\tau))(1+b v(t-\tau))}-ay(t)-py(t)z(t),\\
v'(t)&= ky(t)-u v(t),\\
z'(t)&= cy(t)z(t)-b z(t).
\end{align}

They construct a Lyapunov functional to establish the global dynamics of the system.\par 
More recently, in \cite{yang2015stability}, the authors consider the system

\begin{align}
x'(t)&=s-dx(t)-\frac{\beta x(t)v(t)}{1+ax(t)+bv(t)},\\
y'(t)&=\frac{\beta x(t)v(t)}{1+ax(t)+bv(t)}-ay(t)-py(t)z(t),\\
v'(t)&= ky(t)-u v(t),\\
z'(t)&= cy(t-\tau) z(t-\tau)-b z(t).
\end{align}

They give also results of global stability as well as Hopf bifurcation results.

Yang and Wei in \cite{yang2015analyzing} consider a more general incidence rate, they study the system

\begin{align}
x'(t)&=s-dx(t)-x(t) f(v(t)),\\
y'(t)&= x(t) f(v(t))-ay(t)-py(t)z(t),\\
v'(t)&= ky(t)-u v(t),\\
z'(t)&= cy(t)z(t)-b z(t),
\end{align}

giving some results about global stability in terms of the basic reproduction number and the immune response reproduction number. Here the $f(v)$ is assumed to be a continuous function on $v$ that belongs to $(0, \infty)$ and satisfies $f(0)= 0, f'(v)>0$ for all $v$ greater or equal to 0 and $f''(v)<0$ for all $v$ greater or equal to $0$. \par 
In \cite{hattaf2012mathematical} the authors consider a more general incidence rate $f(x,y,v)v$, where $f$ is assumed to be continuously differentiable in the interior of $\mathbb{R}^{3}_{+}$ and satisfies the following hypotheses
\begin{itemize}
\item[i)]$f(0,y,v)=0$, for all  $y\geq 0$,  $v\geq 0$.
\item[ii)]$\frac{\partial f}{\partial x}(x,y,v) >0$ for all $x>0$, $y> 0$, $v> 0$.
\item[iii)]$\frac{\partial f}{\partial y}(x,y,v)\leq 0$  and $\frac{\partial f}{\partial v}(x,y,v)\leq 0$ for all $x\geq 0$, $y\geq 0$, $v\geq 0$.
\end{itemize}

Finally, in \cite{hattaf2012global} analyse an infection model by virus , with general incidence and immune response, which generalizes the systems by  \cite{xu2012global, hattaf2012mathematical, wang2011global}:

\begin{align*}
\dot{x} &= s-dx-f(x,y,v)v, \\
\dot{y} &= f(x,y,v)v-ay-pyz,\\
\dot{v} &= ky-uv, \\
\dot{z} &= cyz-bz,
\end{align*}

where $f(x,y,v)$ is a continuous and differentiable function in the interior of $ \mathbb{R}_+^{3} $ and satisfies the conditions [i)], [ii) [iii)]. \par 
Another viral infection model is the studied in \cite{shu2013global} given by:

\begin{align*}
\dot{x} &= n(x)-h(x,v), \\
\dot{y} &= \int_0^{\infty} f_1( \tau )h(x(t- \tau),v(t- \tau)) d \tau-ag_1(y)-pw(y,z),\\
\dot{v} &= k \int_0^{\infty} f_2(\tau) g_1(y(t- \tau)) d \tau, \\
\dot{z} &= c \int_0^{\infty} f_3(\tau) w(y(t-\tau),z(t-\tau)) d\tau -bg_3(z).
\end{align*}

Where $x,y,v,z$ denotes the non infected cells, infected cells, virus and specific virus CTL at time $t$, respectively. Conditions on functions $f_i$,$h,w,$ and $g_i$ are specified in \cite{shu2013global}. A related work is \cite{ji2016global}. \par

Based on the discussion above, we will study a delayed viral infection model with general incidence rate and CTL immune response given by

\begin{align} \label{viral11} 
\dot x&= n(x)-f(x(t),y(t),v(t))v(t),\\
\dot y&= \int^{\infty}_{0}f_{1}(\tau)f(x(t-\tau),y(t-\tau),v(t-\tau))v(t-\tau)e^{-\alpha_1 \tau}d\tau -a\varphi_{1}(y(t))-pw(y(t),z(t)), \\
\dot v&= k\int^{\infty}_{0}f_{2}(\tau)e^{-\alpha_{2}\tau}\varphi_{1}(y(t-\tau))d \tau-uv(t),\\
\dot z&= c\int^{\infty}_{0}f_{3}(\tau)w(y(t-\tau),z(t-\tau))d \tau -b\varphi_{2}(z(t)).
 \end{align}

The dynamics of uninfected cells, $x$, in absence of infection is governed by
$x'=n(x),$
where $n(x)$ is the intrinsic growth rate of uninfected cells accounting for both productions
and natural mortality, which is assumed to satisfy the following:
\begin{itemize}
\item[$H_{1})$]$n(x)$ is continuously differentiable, and exist $\bar{x}>0$ such as that $n(\bar{x})=0$,
$n(x)>0$ for $x\in [0,\bar{x})$, and $n(x)<0$ for $x<\bar{x}$. Typical functions appearing in the literature are $n(x)=s-dx$ and $n(x)=s-dx+rx\left(1-{x}/{x_{max}}\right)$.
\end{itemize}
\begin{itemize}
\item[$H_{2})$] $\varphi_{i}$ is strictly increasing on $[0,\infty)$; $\varphi_{i}(0)=0$; $\varphi_{i}'(0)=1$;$\lim_{y\longrightarrow \infty}\varphi_{i}(y)=+\infty$; and exist $k_{i}>0$ such as $\varphi_{i}(y)\geq k_{i}y$ for any $y_{i}\geq 0$.
\end{itemize}

\begin{itemize}
\item[$H_{3})$] $w(y,z)$ is continuously differentiable; $\frac{\partial w(y,z)}{\partial z}>0$ for $y\in (0, \infty)$, $z\in [0, \infty)$; $w(y,z)>0$ for $y \in (0,\infty)$ and $z \in (0,\infty)$ with $w(y,z)=0$ if and if only $y=0$ or $z=0$.
\end{itemize}
All parameters are nonnegative and the distributions $f_i$ for $i=1,2$ are assumed to satisfy the following (see \cite{shu2013global} and \cite{wang2016threshold} ) :
\begin{itemize}
\item $ f_i(\tau) \geq 0, \text{for} \hspace{0.1in} \tau \geq 0$.
\item $ \int_0^{\infty}f_i(\tau)d\tau=1$ for $i=1,2$.
\item $\int_0^{\infty} f_3(\tau)d\tau \leq 1$ and $\int_0^{\infty} f_3(\tau) e^{s \tau} d\tau< \infty$ for some $s>0$. 
\end{itemize}

However, the uniqueness and Global stability results on the positive equilibrium require the following assumption.

\begin{itemize}
\item[$H_{4})$] $w(y,z)=\varphi_{1}(y)\varphi_{2}(z)$. 
\end{itemize}
Motivated by the work mentioned above we consider the following model 

\begin{align} \label{viral1} 
\dot x&= n(x)-f(x(t),y(t),v(t))v(t),\\
\dot y&= \int^{\infty}_{0}f_{1}(\tau)f(x(t-\tau),y(t-\tau),v(t-\tau))v(t-\tau)e^{-\alpha_1 \tau} d \tau -a\varphi_{1}(y(t))-p\varphi_{1}(y(t))\varphi_{2}(z(t)), \\
\dot v&= k\int^{\infty}_{0}f_{2}(\tau)e^{-\alpha_{2}\tau}\varphi_{1}(y(t-\tau))d \tau-uv(t),\\
\dot z&= c\int^{\infty}_{0}f_{3}(\tau)\varphi_{1}(y(t-\tau))\varphi_{2}(z(t-\tau)) d\tau-b\varphi_{2}(z(t)).
\end{align} 

In this paper, we will study the global dynamics model of (\ref{viral1}), organization is as  follows: in section 2, we prove the existence and uniqueness of the infection free equilibrium, the CTL-inactivated infection equilibrium and the CTL-activated infection equilibrium. In section 3, the conditions that allow the  global stability of each equilibrium are determined and proved. Section 4 provide several numerical simulations that shows the results obtained in section 3 and 4. Finally, in section 5 we summarize the results obtained, comparing them with the previous models studied in literature, and setting the guidelines for possible future work.

\subsection{Positivity and Boundedness}
For system \eqref{viral1} the suitable space is $\mathcal{C}^{4}= \mathcal{C} \times \mathcal{C} \times \mathcal{C} \times \mathcal{C} $, where $ \mathcal{C} $ is the Banach space of fading memory type (\cite{fv1988determining}):

\begin{equation*}
C:= \lbrace  \phi \in C((-\infty,0], \mathbb{R}), \hspace{0.1 in}  \phi(\theta)e^{\alpha \theta} \hspace{0.1 in} \text{is uniformly continous for} \hspace{0.1 in}  \theta \in (- \infty,0] \hspace{0.1 in}  \text{and} \hspace{0.1 in}  \Vert \phi\Vert< \infty \rbrace,
\end{equation*}

where $\alpha>0$ is a constant and the norm of a $ \phi \in \mathcal{C}$ is defined as $ \Vert \phi \Vert = \sup_{\theta \leq 0} |\phi(\theta)|e^{\alpha \theta}.$ The nonnegative cone of $\mathcal{C}$ is defined as $\mathcal{C}_+ = C((-\infty,0], \mathbb{R}_+)$.

\begin{theorem}\label{PosB}
Under the initial conditions, all solutions of system (\ref{viral1}) are positive and ultimately uniformly bounded in $X$.
\end{theorem}
\begin{proof}
To see that $x(t)$ is positive, we proceed by contradiction. Let $t_{1}$ the first value
of time such that $x(t_{1})=0$. From the first equation of (\ref{viral1}) we see that
$x'(t_{1})=n(0)>0$ and $x(t_{1})=0$, therefore there exists $\epsilon>0$ such
that $x(t)<0$ for $t \in (t_{1}-\epsilon,t_{1})$, this leads to a contradiction.
It follows that $x(t)$ is always positive. With a similar argument
we see that $y(t)$, $v(t)$ and  $z(t)$ are positive for $t\geq 0$.
\newline\newline
The hypotheses ($H_{1}$) and first equation of  (\ref{viral1}) imply that $\limsup_{t\longrightarrow \infty}x(t)\leq \bar{x}$.
\newline
From two first equations of \eqref{viral1} and assumption ($H_{2}$), we obtain

\begin{equation}
\int^{\infty}_{0}f_{1}(\tau)e^{-\alpha_1 \tau}x'(t-\tau)d\tau +y'(t)=\int^{\infty}_{0}f_{1}(\tau)e^{-\alpha_1 \tau}n(x(t-\tau))d\tau-a\varphi_{1}(y)\leq M_{1}G_{1}-ak_{1}y,
\end{equation}

where $M_{1}=\sup_{x \in [0,\bar{x}]}n(x)$ and $G_{1}=\int^{\infty}_{0}f_{1}(\tau)e^{-\alpha_1 \tau}d\tau $.
\newline\newline
Let $e(t)=\int^{\infty}_{0}f_{1}(\tau)e^{-\alpha_1 \tau}x(t-\tau)d\tau$. Then for $e(t)\leq \bar{x}G_{1}$ for $t>0$. Then

\begin{align*}
(e(t)+y(t))' &\leq M_{1}G_{1}-ak_{1}y\\
&= M_{1}G_{1}+M_{1}G_{1}-M_{1}G_{1}-ak_{1}y\\
&\leq 2M_{1}G_{1}-\frac{M_{1}}{\bar{x}}e(t)-ak_{1}y\\
&\leq 2M_{1}G_{1}-\bar{\mu}(e(t)+y(t))\textnormal{ where }\bar{\mu}=\min\left\lbrace \frac{M_{1}}{\bar{x}},ak_{1}\right\rbrace ,
\end{align*}

and thus $\limsup_{t\longrightarrow \infty}(e(t)+y(t))\leq \frac{2M_{1}G_{1}}{\bar{\mu}}$. Since $e(t)\geq 0$, we know that $\limsup_{t\longrightarrow \infty}y(t) \leq \frac{2M_{1}G_{1}}{\bar{\mu}}$. From third equation of (\ref{viral1}):

\begin{align*}
\dot v&= k\int^{\infty}_{0}f_{2}(\tau)e^{-\alpha_{2}\tau}\varphi_{1}(y(t-\tau)) d\tau  -uv(t)\\
&\leq kM_{2}G_{2}-uv(t),
\end{align*}

where $M_{2}=\sup_{y \in \left[ 0,\frac{2M_{1}G_{1}}{\bar{\mu}}\right] }\varphi_{1}(y)$ and $G_{2}=\int^{\infty}_{0}f_{2}(\tau)e^{-\alpha_2\tau}d\tau$,
and thus $\limsup_{t\longrightarrow \infty}v(t)\leq \frac{k M_{2}G_{2}}{u}$.
\newline\newline
Using an argument similar, let $G_3=\int^{\infty}_{0}f_{3}(\tau)d \tau  $, $L(t)=\frac{cG_{3}}{p}y(t)+z(t)$, $\tilde{\mu}=\min\lbrace ak_{1},bk_{2}\rbrace$\newline and $M_{3}=\sup_{(x,y,v) \in [0,\bar{x}]\times [0, \frac{2M_{1}G_{1}}{\bar{\mu}}]\times[0,\frac{k M_{2}G_{2}}{u}] }f(x,y,v)v$, then

\begin{align*}
L'(t)&=\frac{cG_{3}}{p}G_{1}f(x,y,v)v-\frac{acG_{3}}{p}\varphi_{1}(y)-cG_{3}\varphi_{1}(y)\varphi_{2}(z)+cG_{3}\varphi_{1}(y)\varphi_{2}(z)-b\varphi_{2}(z)\\
&\leq \frac{cG_{3}G_{1}}{p}M_{3}-\frac{acG_{3}}{p}k_{1}y-bk_{2}z\\
&\leq \frac{cG_{3}G_{1}}{p}M_{3}-\tilde{\mu}\left(\frac{cG_{3}}{p}y+z \right)\\
&=\frac{cG_{3}G_{1}}{p}M_{3}-\tilde{\mu}L(t),
\end{align*}

therefore $\limsup_{t\longrightarrow \infty}L(t)\leq \frac{cG_{1}G_{3}M_{3}}{p\tilde{\mu}}$. Since $\frac{cG_{3}}{p}y(t)\geq 0$, we know that $\limsup_{t\longrightarrow \infty}z(t) \leq \frac{cG_{1}G_{3}M_{3}}{p\tilde{\mu}}$.
\newline
Therefore, $x(t)$, $y(t)$, $v(t)$ and $z(t)$ are ultimately uniformly bounded in $\mathcal{C}\times \mathcal{C}\times \mathcal{C} \times \mathcal{C}$.
\end{proof}
Theorem (\ref{PosB}) implies that omega limit set of system (\ref{viral1}) are contained in the
following bounded feasible region:

\begin{equation*}
\Gamma=\left\lbrace (x,y,v,z)\in \mathcal{C}^{4}_{+}:\parallel x \parallel\leq \bar{x},\parallel y\parallel\leq  \frac{2M_{1}G_{1}}{\bar{\mu}}, \parallel v \parallel \leq \frac{k M_{2}G_{2}}{u}, \parallel z\parallel \leq \frac{cG_{1}G_{3}M_{3}}{p\tilde{\mu}} \right\rbrace .
\end{equation*}

It can be verified that the region $\Gamma$ is positively invariant with respect to model (\ref{viral1})
and that the model is well posed.
\newline\newline

\section{Existence and uniqueness of equilibria of system}

At any equilibrium we have

\begin{align}
n(x)-f(x,y,v)v&=0,\label{1}\\
f(x,y,v)v-\frac{a}{G_{1}}\varphi_{1}(y)-\frac{p}{G_{1}}\varphi_{1}(y)\varphi_{2}(z)&=0,\label{2}\\
k\varphi_{1}(y)-\frac{u}{G_{2}}v&=0, \label{3}\\
c\varphi_{1}(y)\varphi_{2}(z)-\frac{b}{G_{3}}\varphi_{2}(z)&=0.\label{4}
\end{align}

The system (\ref{viral1}) always has an infection free equilibrium $E_{0}=(\bar{x},0,0,0)$. In addition to $E_{0}$ the system could have two types of Chronic infection Equilibria $E_{1}=(x_{1},y_{1},v_{1},0)$ and $E_{2}=(x_{2},y_{2},v_{2},z_{2})$ in $\Gamma$ where the entries of  $E_{1}$ and $E_{2}$ are strictly positives. The equilibria $E_{1}$ and $E_{2}$ are called CTL-inactivated infection equilibrium (CTL-IE) and CTL-activated infection equilibrium (CTL-AE), respectively.\par 

We define the general reproduction number as

\begin{equation*}
R(x,y,v)=\frac{kG_{1}G_{2}f(x,y,v)}{au},
\end{equation*}

which is the ratio of the per capita production and decay rates of mature viruses at an equilibrium
$(x,y,v,z)$ with $z=0$. In particular, at the infection free equilibrium, $E_{0}$, we denote $R(\bar{x},0,0)$ by $R_{0}$, representing the basic production number for viral infection:

\begin{align}
R_{0}=R(\bar{x},0,0)=\frac{kG_{1}G_{2}f(\bar{x},0,0)}{au}.
\end{align}

From assumption (H2) we have that $\varphi_1$ is invertible, so we can define from equation \eqref{4}:

\begin{equation*}
\hat{y}=\varphi_1^{-1}\left( \frac{b}{cG_3} \right), \quad \hat{v} = \frac{k \varphi_1(\hat{y})G_2}{u}.
\end{equation*}

 Define also $H(x)=n(x)-f(x,\hat{y},\hat{v})\hat{v},$ we have $H(0)=n(0)>0$ and $H(\bar{x})=- f(\bar{x},\hat{y},\hat{v})\hat{v} $, with $f(\bar{x},\hat{y},\hat{v})>f(0,\hat{y},\hat{v})=0$ by ($ii$), so there exists a $\hat{x} \in (0, \bar{x})$ such that $H(\hat{x})=0$. We denote:

 \begin{equation}
 R_1=R(\hat{x}, \hat{y}, \hat{v}),
 \end{equation}

 and refer it as the viral reproduction number. From assumptions on $f$ it is easy to see that $f(\bar{x},0,0)>f (\bar{x},y,v)$, for all $y,v>0$, moreover for all $x \in [0, \bar{x})$ we have $f(x,y,v)<f(\bar{x},y,v)$, so:

 \begin{equation}
 R_0= R(\bar{x},0,0)>R(\bar{x},y,v)>R(x,y,v), \quad \forall x \in [0, \bar{x}),y,v>0.
 \end{equation}

 Particularly, $R_0>R_1$. The basic reproduction number for the CTL response is given by:

\begin{equation*}
R_{CTL}= \frac{cG_{3}\varphi_{1}(y_{1})}{b}.
\end{equation*}

For proof of the existence and uniqueness of equilibria, we require two additional assumption. First, we define the following sets:

\begin{align*}
X_{n}&=\{ \xi\in [0,\bar{x}]:(n(x)-n(\xi))(x-\xi)<0 \textnormal{ for } x\neq \xi, x\in [0,\bar{x}]  \},\\
X_{f}(y,v)&=\{ \xi\in [0,\bar{x}]:(f(x,y,v)-f(\xi,y,v))(x-\xi)<0 \textnormal{ for } x\neq \xi, x\in [0,\bar{x}]  \},\\
X&=\cap_{y,v\in (0,\bar{y})\times (0,\bar{v})}X_{f}(y,v)\cap X_{n}.
\end{align*}

The following conditions are used to guarantee the uniqueness of the equilibria.
\begin{itemize}
\item[$(A_{1})$]The system (\ref{viral1}) has a $E_{1}=(x_{1},y_{1},v_{1},0)$ satisfying $x_{1}\in X$.
\item[$(A_{2})$]The system (\ref{viral1}) has a $E_{2}=(x_{2},y_{2},v_{2},z_{2})$ satisfying $X_{n}\cap X_{f}(y_{2},v_{2})$.
\end{itemize}
\begin{theorem}
Assume that $ i)-iii)$ and $H_1-H_{4}$ are satisfied.
\begin{enumerate}
\item If $R_{0}\leq 1$, then $E_0=(x_0,0,0)$ is the unique equilibrium of  the system(\ref{viral1}).
\item if $R_{1}\leq 1< R_{0}$ then in addition to $E_0$, system (\ref{viral1}) has a CTL inactivated infection equilibrium $E_1$.
\item If $R_0>R_{1}>1$ then, in addition to $E_0$ and $E_1$ system (\ref{viral1}) has a CTL activated infection equilibrium.
\end{enumerate}  \label{theo2}
\end{theorem}

When $x=\bar{x}$ and $y=v=z=0$ the equations (\ref{1})-(\ref{4})are satisfied,
therefore $E_{0}=(x_{0},0,0,0)$ is a steady state called the infection free
equilibrium. To proof that it is unique when $R_0<1$, we look for the existence of a positive equilibrium. \par 

To find a positive equilibrium  we proceed as follows:
from equation (\ref{4}) $\varphi_{2}(z)=0$ or $\varphi_{1}(y)=\frac{b}{cG_{1}}$. If $\varphi_{2}(z)=0$ then from assumption $(H_{2})$ we get, $z=0$ and using (\ref{1})-(\ref{3})

\begin{align}\label{fused}
n(x)=f(x,y,v)v=\frac{a\varphi_{1}(y)}{G_{1}}=\frac{auv}{kG_{1}G_{2}}.
\end{align}

By ($H_{2}$), we know that $\varphi^{-1}_{1}$ exists. Solving $n(x)=\frac{a\varphi_{1}(y)}{G_{1}}$ for $y$ gives us that

\begin{equation*}
y(x)=\phi(x)=\varphi^{-1}_{1}\left(\frac{n(x)G_{1}}{a}\right),
\end{equation*}

with $\phi(\bar{x})=0$ and $\phi(0)=y^{0}$  is a unique root of equations $n(0)=\frac{a\varphi_{1}(y^{0})}{G_{1}}$. Solving (\ref{3}) for $v$ we obtain:

\begin{equation*}
v(x)= \frac{k n(x)G_1G_2}{au}.
\end{equation*}

Note that, from ($ii$) we have $f(x,y,v)>f(0,y,v)=0$, $\forall x>0$. So if $E^{*}=(x^{*},y^{*},v^{*},z^{*})$ is a positive equilibrium, then $n(x^{*})=f(x^{*},y^{*},v^{*})v^{*}>0$, so $x^{*} \in (0, \bar{x})$. \par 
Now, using \eqref{fused} define on the interval $[0,\bar{x})$ the function  $G$, 

\begin{align*}
G(x)&= f(x,y(x),v(x))- \frac{n(x)}{v(x)}, \\
&=f\left( x,\phi(x),\frac{n(x)kG_{1}G_{2}}{au} \right)  - \frac{au}{kG_1 G_2},
\end{align*}

we have $G(0)=-\frac{au}{kG_1 G_2} <0$, $G(\bar{x})=f(\bar{x},0,0)-\frac{au}{kG_{1}G_{2}}=\frac{au}{kG_{1}G_{2}}\left(R_{0}-1\right).$ Therefore, if $R_{0} > 1$, then $G(x)$ has a root $x^{*}\in (0, \bar{x})$ such that 

\begin{equation*}
 f(x^{*},y(x^{*}),v(x^{*}))- \frac{n(x^{*})}{v(x^{*})}=0, 
\end{equation*}

or equivalently

\begin{equation*}
n(x^{*})- f(x^{*},y(x^{*}),v(x^{*}))v(x^{*})=0.
\end{equation*}

We conclude that, for $R_0>1$  there exists another equilibrium $E_{1}=(x_{1},y_{1},v_{1},0)$ with $x_{1}=x^{*}\in (0,\bar{x})$,  $y_{1}=\phi(x_1)$  and $v_{1}=\frac{kG_{2}\phi(x_{1})}{u}$. Moreover, using the fact that $R_0>R(x,y,v)$, for all $x,y,v>0$ , when $R_0<1$ we have $R(x,y,v)<1$. Using \eqref{fused} we arrive to

\begin{equation*}
n(x)-f(x,y,v)v>0, \quad \forall x,y,v>0,
\end{equation*}

so equation \eqref{1} never holds. Therefore, $E_1$ exists iff $R_0>1$.
\par 
Next we show that $E_{1}=(x_{1},y_{1},v_{1},0)$ in unique. Suppose, to the contrary, there exists another CTL-IE equilibrium $E^{\ast}_{1}=(x^{\ast}_{1},y^{\ast}_{1},v^{\ast}_{1},0)$. Without of loss generality, we assumed that $x^{\ast}_{1}<x_{1}$. Then $x_{1}\in X_{n}$ implies that $n(x^{\ast}_{1})>n(x_{1})$. By virtue $x_{1}\in X_{h}$, we get $f(x_{1}^{\ast},y_{1}^{\ast},v_{1}^{\ast})<f(x_{1},y_{1},v_{1})$. On other hand, it follows from (\ref{fused}) $f(x_{1}^{\ast},y_{1}^{\ast},v_{1}^{\ast})=f(x_{1},y_{1},v_{1})={au}/{kG_{1}G_{2}}$. This a contradiction, and thus $E_{1}$ is the unique CTL-IE.
\par 
Note that the CTL-AE $E_{2}=(x_{2},y_{2},v_{2},z_{2})$ exists if $(x_{2},y_{2},v_{2},z_{2}) \in \mathbf{R}^{4}_{+}$ satisfies the equilibrium equations \eqref{1}-\eqref{4} and $\varphi_1(y)={b}/{cG_3}$. The four equations of system (\ref{viral1}) and ($H_{2}$) imply that

\begin{align}
y_{2}=\varphi^{-1}_{1} \left( \frac{b}{cG_{3}} \right).
\end{align}

Note that the values $\hat{x},\hat{y},\hat{v}$ used to define $R_1$ clearly satisfy the equilibrium equations. Solving the  equation \eqref{2} for $z$ and using (H2) yields $\hat{z}=\phi_2^{-1} \left( \frac{kG_{1}G_{2}f(\hat{x},\hat{y},\hat{v})-au}{pu} \right) =\phi_2^{-1} \left( \frac{a\left(R_1-1)\right) }{p} \right) $. Using the fact that $\varphi_2^{-1}$ is defined on $[0, \infty)$ we conclude that the CTL-AE $E_{2}=(\hat{x},\hat{y},\hat{v}, \hat{z})$ exists if and only if $R_1>1$.\par 

Now we will prove that $E_{2}$ is unique. Suppose there exists another CTL-AE, $E^{\ast}_{2}=(x_{2}^{\ast},y_{2}^{\ast},v_{2}^{\ast},z_{2}^{\ast})$. Then $y_{2}=y_{2}^{\ast}$ and $v_{2}=v_{2}^{\ast}$. Without of generality, we assume that $x_{2}^{\ast}< x_{2}$. Then $x_{2}\in X_{n}$ implies $n(x_{2}^{\ast})>n(x_{2})$. Note that $n(x_{2}^{\ast})=f(x_{2}^{\ast},y_{2},v_{2})$ and $n(x_{2})=f(x_{2},y_{2},v_{2})$, implies that $f(x_{2}^{\ast},y_{2},v_{2})> f(x_{2},y_{2},v_{2})$. This contradicts that $x_{2}\in X_{f}(y_{2},v_{2})$ and hence $E_{2}$ is unique.

\section{Global stability}
Let $R_0$ and $R_1$ be defined as in previous section.

\begin{theorem}
If $R_{0}<1$, then infection free-Equilibrium $E_{0}$ of model (\ref{viral1}) is globally asymptotically stable. \label{theo3}
\end{theorem}
\begin{proof}
Define a Lyapunov functional

\begin{align*}
V_{1}&=x-x_{0}+\int^{x}_{x_{0}}\frac{f(x_{0},0,0)}{f(s,0,0)}ds+\frac{1}{G_{1}} y+\frac{a}{kG_{1}G_{2}}v+\frac{p}{cG_{1}G_{3}}z\\
&+\frac{1}{G_{1}}\int^{\infty}_{0}f_{1}(\tau) e^{-\alpha_1 \tau }\int^{t}_{t-\tau}f(x(s),y(s),v(s))v(s)dsd\tau+\frac{a}{G_{1}G_{2}}\int^{\infty}_{0}f_{2}(\tau)e^{-\alpha_{2}\tau}\int^{t}_{t-\tau}\varphi_{1}(y(s))dsd\tau\\
&+\frac{p}{G_{1}G_{3}}\int^{\infty}_{0}f_{3}(\tau)\int^{t}_{t-\tau}\varphi_{1}(y(s))\varphi_{2}(z(s))dsd\tau,
\end{align*}

where $G_{i}=\int^{\infty}_{0}f_{i}(\tau)$,$i=1,2,3$.

Calculating the derivative of $V_{1}$ along the positive solution of system (\ref{viral1}), it follows that

\begin{align*}
\dot V_1&=\left(1-\frac{f(x_{0},0,0)}{f(x,0,0)}\right)\dot{x}+\frac{1}{G_{1}}\dot{y}+\frac{a}{kG_{1}G_{2}}\dot{v}+\frac{p}{cG_{1}G_{3}}\dot{z}\\
&+f(x,y,v)v-\frac{1}{G_{1}}\int^{\infty}_{0}f_{1}(\tau)e^{-\alpha_{1}\tau}f(x(t-\tau),y(t-\tau),v(t-\tau))d\tau\\
&+\frac{a}{G_{1}}\varphi_{1}(y)-\frac{a}{G_{1}G_{2}}\int^{\infty}_{0}f_{2}(\tau)e^{-\alpha_{2}\tau}\varphi_{1}(y(t-\tau))d\tau\\
&+\frac{p}{G_{1}}\varphi_{1}(y)\varphi_{2}(z)-\frac{p}{G_{1}G_{3}}\int^{\infty}_{0}f_{3}(\tau)\varphi_{1}(y(t-\tau))\varphi_{2}(z(t-\tau))d\tau\\
&=\left(1-\frac{f(x_{0},0,0)}{f(x,0,0)}\right)\left(n(x)-f(x,y,v)v\right)\\
&+\frac{1}{G_{1}}\left( \int^{\infty}_{0}f_{1}(\tau)e^{-\alpha_{1}\tau}f(x(t-\tau),y(t-\tau),v(t-\tau))v(t-\tau)-a\varphi_{1}(y)-p\varphi_{1}(y)\varphi_{2}(z)\right)\\
&+\frac{a}{kG_{1}G_{2}}\left(k\int^{\infty}_{0}f_{2}(\tau)e^{-\alpha_{2}\tau}\varphi_{1}(y(t-\tau))-uv\right)+\frac{p}{cG_{1}G_{3}}\left(c\int^{\infty}_{0}f_{3}(\tau)\varphi_{1}(y(t-\tau)\varphi_{2}(z(t-\tau))-b\varphi_{2}(z)\right)\\
&+f(x,y,v)v-\frac{1}{G_{1}}\int^{\infty}_{0}f_{1}(\tau)e^{-\alpha_{1}\tau}f(x(t-\tau),y(t-\tau),v(t-\tau))d\tau\\
&+\frac{a}{G_{1}}\varphi_{1}(y)-\frac{a}{G_{1}G_{2}}\int^{\infty}_{0}f_{2}(\tau)e^{-\alpha_{2}\tau}\varphi_{1}(y(t-\tau))d\tau\\
&+\frac{p}{G_{1}}\varphi_{1}(y)\varphi_{2}(z)-\frac{p}{G_{1}G_{3}}\int^{\infty}_{0}f_{3}(\tau)\varphi_{1}(y(t-\tau))\varphi_{2}(z(t-\tau))d\tau.\\
\end{align*}

Using $n(x_{0})=0$ and simplifying, we get

\begin{align*}
\dot V_1&=\left(n(x)-n(x_{0})\right)\left(1-\frac{f(x_{0},0,0)}{f(x,0,0)}\right)+\frac{au}{kG_{1}G_{2}}v\left(\frac{f(x,y,v)}{f(x,0,0)}\frac{f(x_{0},0,0)kG_{1}G_{2}}{au}-1\right)-\frac{pb}{cG_{1}G_{3}}\varphi_{2}(z)\\
&=\left(n(x)-n(x_{0})\right)\left(1-\frac{f(x_{0},0,0)}{f(x,0,0)}\right)+\frac{au}{kG_{1}G_{2}}v\left(\frac{f(x,y,v)}{f(x,0,0)}R_{0}-1\right)-\frac{pb}{cG_{1}G_{3}}\varphi_{2}(z)\\
&\leq \left(n(x)-n(x_{0})\right)\left(1-\frac{f(x_{0},0,0)}{f(x,0,0)}\right)+\frac{au}{kG_{1}G_{2}}v\left(R_{0}-1\right)-\frac{pb}{cG_{1}G_{3}}\varphi_{2}(z).
\end{align*}

Using the  following inequalities:

\begin{align*}
n(x)-n(x_{0})< 0,  \ \ 1-\frac{f(x_{0},0,0)}{f(x,0,0)} &\geq 0 \ \ for \ x\geq x_{0},\\
n(x)-n(x_{0})> 0,  \ \ 1-\frac{f(x_{0},0,0)}{f(x,0,0)} &\leq 0 \ \ for \ x< x_{0}.
\end{align*}

We have that

$$\left(n(x)-n(x_{0})\right)\left(1-\frac{f(x_{0},0,0)}{f(x,0,0)}\right)\leq 0.$$

Since $R_{0}\leq 1$, we have $\dot{V}_1\leq 0$. Therefore the disease free $E_{1}$ is stable, $\dot{V}_1=0$ if  and only if $x=x_{0}, y=0, v=0, z=0$. So, the largest compact invariant set in $\{(x,y,v,z):\dot{V}_1=0 \}$ is just the singleton $E_{1}$. From LaSalle invariance principle , we conclude that $E_{1}$ is globally asymptotically stable.

\end{proof}
\begin{theorem}
If $R_{0}>1$ and $R_{1}<1$, then CTL-IE $E_{1}$, of model (\ref{viral1}) is globally asymptotically stable. \label{theo4}
\end{theorem}
\begin{proof}
Consider the following Lyapunov functional:

$$V_{1}=\hat{L}(t)$$

where

\begin{align*}
\hat{L}&=x-x_{1}-\int^{x}_{x_{1}}\frac{f(x_{1},y_{1},v_{1})}{f(s,y_{1},v_{1})}ds+\frac{1}{G_{1}}\int^{y}_{y_{1}}\left(1-\frac{\varphi_{1}(y_{1})}{\varphi_{1}(\sigma)}\right)d\sigma+\frac{a}{kG_{1}G_{2}}\int^{v}_{v_{1}}\left(1-\frac{v_{1}}{\sigma}\right)d\sigma\\
&+\frac{p}{cG_{1}G_{3}}z\\
&+\frac{f(x_{1},y_{1},v_{1})v_{1}}{G_{1}}\int^{\infty}_{0}f_{1}(\tau)e^{-\alpha_{1}\tau}\int^{\tau}_{0}H\left(\frac{f(x(t-\omega),y(t-\omega),v(t-\omega))v(t-\omega)}{f(x_{1},y_{1},v_{1})v_{1}}\right)d\omega d\tau\\
&+ \frac{a\varphi_{1}(y_{1})}{G_{1}G_{2}}\int^{\infty}_{0}f_{2}(\tau)e^{-\alpha_{2}\tau}\int^{\tau}_{0}H\left( \frac{\varphi_{1}(t-\omega)}{\varphi_{1}(y_{1})}\right)d\omega d\tau\\
&+\frac{p}{G_{1}G_{3}}\int^{\infty}_{0}f_{3}(\tau)\int^{\tau}_{0}\varphi_{1}(y(t-\omega))\varphi_{2}(z(t-\omega))d\omega d\tau.
\end{align*}

At infected equilibrium

\begin{align}
n(x_{1})-f(x_{1},y_{1},v_{1})v_{1}&=0,\label{ig1}\\
f(x_{1},y_{1},v_{1})v_{1}&=\frac{a\varphi_{1}(y_{1})}{G_{1}} ,\\
\frac{u}{kG_{2}}&=\frac{\varphi_{1}(y_{1})}{v_{1}}.\label{ig3}
\end{align}

Calculating the derivative of $\hat{L}$ a long positive solutions of (\ref{viral1}), we get

\begin{align*}
\dot{V}_1&=\left(1-\frac{f(x_{1},y_{1},v_{1})}{f(x,y_{1},v_{1})}\right)\dot{x}+\frac{1}{G_{1}}\left(1-\frac{\varphi_{1}(y_{1})}{\varphi_{1}(y)}\right)\dot{y}+\frac{a}{kG_{1}G_{2}}\left(1-\frac{v_{1}}{v}\right)\dot{v}+\frac{p}{cG_{1}G_{2}}\dot{z}\\
&+\frac{f(x_{1},y_{1},v_{1})v_{1}}{G_{1}}\int^{\infty}_{0}f_{1}(\tau)e^{-\alpha_{1}\tau}\\
&\left(H\left(\frac{f(x,y,v)v}{f(x_{1},y_{1},v_{1})v_{1}}\right)-H\left(\frac{f(x(t-\tau),y(t-\tau),v(t-\tau))v(t-\tau)}{f(x_{1},y_{1},v_{1})v_{1}}\right)\right)d\tau \\
&+ \frac{a\varphi_{1}(y_{1})}{G_{1}G_{2}}\int^{\infty}_{0}f_{2}(\tau)e^{-\alpha_{2}\tau}\left( H\left( \frac{\varphi_{1}(y)}{\varphi_{1}(y_{1})}\right)-H\left(\frac{\varphi_{1}(y(t-\tau))}{\varphi_{1}(y_{1})}\right)\right) d\tau\\
&+\frac{p}{G_{1}}\varphi_{1}(y)\varphi_{2}(z)-\frac{p}{G_{1}G_{3}}\int^{\infty}_{0}f_{3}(\tau)\varphi_{1}(y(t-\tau))\varphi_{2}(z(t-\tau)) d\tau \\
\end{align*}

\begin{align*}
&=\left(1-\frac{f(x_{1},y_{1},v_{1})}{f(x,y_{1},v_{1})}\right)\left(n(x)-f(x,y,v)v\right)\\
&+\frac{1}{G_{1}}\left(1-\frac{\varphi_{1}(y_{1})}{\varphi_{1}(y)}\right)\left(\int^{\infty}_{0}f_{1}(\tau)e^{-\alpha_{1}\tau}f(x(t-\tau),y(t-\tau),v(t-\tau))v(t-\tau)-a\varphi_{1}(y)-p\varphi_{1}(y)\varphi_{2}(z)\right)\\
&+\frac{a}{kG_{1}G_{2}}\left(1-\frac{v_{1}}{v}\right)\left(k\int^{\infty}_{0}f_{2}(\tau)e^{-\alpha_{2}\tau}\varphi_{1}(y-\tau)-uv\right)\\
&+\frac{p}{cG_{1}G_{2}}\left(c\int^{\infty}_{0}f_{3}(\tau)\varphi_{1}(y(t-\tau))\varphi_{2}(z(t-\tau)) -b\varphi_{2}(z)\right)\\
&+\frac{f(x_{1},y_{1},v_{1})v_{1}}{G_{1}}\int^{\infty}_{0}f_{1}(\tau)e^{-\alpha_{1}\tau}\left(H\left(\frac{f(x,y,v)v}{f(x_{1},y_{1},v_{1})v_{1}}\right)-H\left(\frac{f(x(t-\tau),y(t-\tau),v(t-\tau))v(t-\tau)}{f(x_{1},y_{1},v_{1})v_{1}}\right)\right)d\tau \\
&+ \frac{a\varphi_{1}(y_{1})}{G_{1}G_{2}}\int^{\infty}_{0}f_{2}(\tau)e^{-\alpha_{2}\tau}\left( H\left( \frac{\varphi_{1}(y)}{\varphi_{1}(y_{1})}\right)-H\left(\frac{\varphi_{1}(y(t-\tau))}{\varphi_{1}(y_{1})}\right)\right) d\tau\\
&+\frac{p}{G_{1}}\varphi_{1}(y)\varphi_{2}(z)-\frac{p}{G_{1}G_{3}}\int^{\infty}_{0}f_{3}(\tau)\varphi_{1}(y(t-\tau))\varphi_{2}(z(t-\tau)) d\tau. \\
\end{align*}

Using (\ref{ig1})-(\ref{ig3}), we, get

\begin{align*}
\dot{V}_1&=\left(n(x)-n(x_{1})\right)\left(1-\frac{f(x_{1},y_{1},v_{1})}{f(x,y_{1},v_{1})}\right)\\
&+\frac{a\varphi_{1}(y_{1})}{G_{1}}\left(1-\frac{f(x_{1},y_{1},v_{1})}{f(x,y_{1},v_{1})}+\frac{f(x,y,v)v}{f(x,y_{1},v_{1})v_{1}}\right)\\
&+\frac{a\varphi_{1}(y_{1})}{G_{1}}\left(1-\frac{1}{G_{1}}\int^{\infty}_{0}f_{1}(\tau)e^{-\alpha_{1}\tau}\frac{\varphi_{1}(y_{1})}{\varphi_{1}(y)}\frac{v(t-\tau)}{v_{1}}\frac{f(x(t-\tau,y(t-\tau),v(t-\tau)))}{f(x_{1},y_{1},v_{1})}\right)d\tau\\
&+\frac{a\varphi_{1}(y_{1})}{G_{1}}\left[\frac{1}{G_{1}}\int^{\infty}_{0}f_{1}(\tau)e^{-\alpha_{1}\tau}\textnormal{ln}\left(\frac{f(x(t-\tau),y(t-\tau),v(t-\tau))v(t-\tau)}{f(x_{1},y_{1},v_{1})v_{1}}\right)d\tau \right]\\
&+\frac{a\varphi_{1}(y_{1})}{G_{1}}\left[\frac{1}{G_{2}}\int^{\infty}_{0}f_{2}(\tau)e^{-\alpha_{2}\tau}\textnormal{ln}\left(\frac{\varphi_{1}(y(t-\tau))}{\varphi_{1}(y_{1})}d\tau\right)-\textnormal{ln}\left(\frac{f(x,y,v)v\varphi_{1}(y)}{f(x_{1},y_{1},v_{1})v_{1}\varphi_{1}(y_{1})}\right)\right]\\
&+\frac{pb\varphi_{1}(z)}{cG_{1}G_{3}}\left(\frac{cG_{3}\varphi_{1}(y_{1})}{b}-1\right).
\end{align*}

Therefore

\begin{align*}
\dot{V}_1&=\left(n(x)-n(x_{1})\right)\left(1-\frac{f(x_{1},y_{1},v_{1})}{f(x,y_{1},v_{1})}\right)\\
&+\frac{a\varphi_{1}(y_{1})}{G_{1}}\left(1-\frac{v}{v_{1}}+\frac{f(x,y_{1},v_{1})}{f(x,y,v)}+\frac{f(x,y,v)v}{f(x,y_{1},v_{1})v_{1}}\right)\\
&-\frac{a\varphi_{1}(y_{1})}{G_{1}}\left[H\left(\frac{f(x_{1},y_{1},v_{1})}{f(x,y_{1},v_{1})}\right)+H\left(\frac{f(x,y_{1},v_{1})}{f(x,y,v)}\right)\right]\\
&-\frac{a\varphi_{1}(y_{1})}{G_{1}}\left[\frac{1}{G_{1}}\int^{\infty}_{0}f_{1}(\tau)e^{-\alpha_{1}\tau}H\left(\frac{\varphi_{1}(y_{1})}{\varphi_{1}(y)}\frac{v(t-\tau)}{v_{1}}\frac{f(x(t-\tau),y(t-\tau),z(t-\tau))}{f(x_{1},y_{1},v_{1})}\right)d\tau \right]\\
&-\frac{a\varphi_{1}(y_{1})}{G_{1}}\left[\frac{1}{G_{2}}\int^{\infty}_{0}f_{2}(\tau)e^{-\alpha_{2}\tau}H\left(\frac{\varphi_{1}(y(t-\tau))v_{1}}{\varphi_{1}(y_{1})v}\right)d\tau\right]\\
&+\frac{pb\varphi_{1}(z)}{cG_{1}G_{3}}\left(R_{1}-1\right).
\end{align*}

Using the inequalities:

\begin{align*}
n(x)-n(x_{1})<0, \ \  \ 1-\frac{f(x_{1},y_{1},v_{1})}{f(x,y_{1},v_{1})} &\geq 0 \ \ for \ x\geq x_{1},\\
n(x)-n(x_{1})>0, \ \ \  1-\frac{f(x_{1},y_{1},v_{1})}{f(x,y_{1},v_{1})} &\leq 0 \ \ for \ x< x_{1}.
\end{align*}

We have that

$$\left(1-\frac{x}{x_{1}}\right)\left(1-\frac{f(x_{1},y_{1},v_{1})}{f(x,y_{1},v_{1})}\right)\leq 0,$$

\begin{align*}
&-1-\frac{v}{v_{1}}+\frac{f(x,y_{1},v_{1})}{f(x,y,v)}+\frac{v}{v_{1}}\frac{f(x,y,v)}{f(x,y_{1},v_{1})}\\
&=\left(1-\frac{f(x,y,v)}{f(x,y_{1},v_{1})} \right)\left(\frac{f(x,y_{1},v_{1})}{f(x,y,v)}-\frac{v}{v_{1}}\right)\leq 0.
\end{align*}

Since $R_{1}\leq 1$, we have $\dot{V_{1}}\leq 0$, thus $E_{1}$ is stable. $\dot{V_{1}}=0$ if and only if
$x=x_{1}, y=y_{1}, v=_{1}, z=0$. So, the largest compact invariant set in $\{(x,y,v,z):\dot{v_{1}}=0\}$ is
the singleton $E_{1}$. From LaSalle invariance principle, we conclude that $E_{1}$ is globally
asymptotically stable.

\end{proof}
\begin{theorem}
Assume that $i)-H_{4}$ hold and $f_{3}(\tau)=\delta (\tau)$.
If $R_{1}>1$, then the CTL-AE, $E_{2}$ is Globally asymptotically stable. \label{theo5}
\end{theorem}
\begin{proof}
Define a Lyapunov functional for $E_{2}$.

\begin{align*}
V_{2}&=x-x_{2}-\int^{x}_{x_{2}}\frac{f(x_{2},y_{2},v_{2})}{f(s,y_{2},v_{2})}ds+\frac{1}{G_{1}}\int^{y}_{y_{2}}\left(1-\frac{\varphi_{1}(y_{2})}{\varphi_{1}(\sigma)}\right)d\sigma+\frac{a+p\varphi_{2}(z_{2})}{kG_{1}G_{2}}\int^{v}_{v_{2}}\left(1-\frac{v_{2}}{\sigma}\right)d\sigma \\
&+ \frac{p}{cG_{1}}\int^{z}_{z_{2}}\left(1-\frac{\varphi_{2}(z_{2})}{\varphi_{2}(\sigma)}\right)d\sigma\\
&+\frac{1}{G_{1}}f(x_{2},y_{2},v_{2})v_{2}\int^{\infty}_{0}f_{1}(\tau)e^{-\alpha_{1}\tau}\int^{t}_{t-\tau}H\left(\frac{f(x(s),y(s),v(s))v(s)}{f(x_{2},y_{2},v_{2})v_{2}}\right)dsd\tau\\
&+\frac{a+p\varphi_{2}(z_{2})}{kG_{1}G_{2}}\varphi_{1}(y_{2})\int^{\infty}_{0}f_{2}(\tau)e^{-\alpha_{2}\tau}\int^{t}_{t-\tau}H\left(\frac{\varphi_{1}(y(s))}{\varphi_{1}(y_{2})}\right) dsd\tau.
\end{align*}

The derivative of $V_2$ along with the solutions of system (\ref{viral1}) is

\begin{align*}
\dot{V_{2}}&=\left(1-\frac{f(x_{2},y_{2},v_{2})}{f(x,y_{2},v_{2})}\right)\dot{x}+\frac{1}{G_{1}}\left(1-\frac{\varphi_{1}(y_{2})}{\varphi_{1}(y)}\right)\dot{y}+\frac{a+p\varphi_{2}(z_{2})}{kG_{1}G_{2}}\left(1-\frac{v_{2}}{v}\right)\dot{v}\\
&+ \frac{p}{cG_{1}}\int^{z}_{z_{2}}\left(1-\frac{\varphi_{2}(z_{2})}{\varphi_{2}(z)}\right)\dot{z}+\frac{1}{G_{1}}f(x_{2},y_{2},v_{2})v_{2}\int^{\infty}_{0}f_{1}(\tau)e^{-\alpha_{1}\tau}\\
&\left[ H\left(\frac{f(x,y,v)v}{f(x_{2},y_{2},v_{2})v_{2}}\right)-H\left(\frac{f(x(t-\tau),y(t-\tau),v(t-\tau))v(t-\tau)}{f(x_{2},y_{2},v_{2})v_{2}}\right)\right] d\tau\\
&+\frac{a+p\varphi_{2}(z_{2})}{kG_{1}G_{2}}\varphi_{1}(y_{2})\int^{\infty}_{0}f_{2}(\tau)e^{-\alpha_{2}\tau}\left[ H\left(\frac{\varphi_{1}(y)}{\varphi_{1}(y_{2})}\right)-H\left(\frac{\varphi_{1}(y(t-\tau))}{\varphi_{1}(y_{2})}\right)\right]  d\tau.
\end{align*}
Applying  $n(x_{2})=f(x_{2},y_{2},v_{2})v_{2}$, $f(x_{2},y_{2},v_{2})v_{2}=\frac{1}{G_{1}}(a\varphi_{1}(y_{1})+p\varphi_{1}(y_{1})\varphi_{2}(z_{2}))$, $\varphi_{1}(y_{1})=\frac{b}{c}$, $\frac{u}{kG_{2}}=\frac{\varphi_{1}(y_{2})}{v_{2}}$, we obtain
\begin{align*}
\dot{V_{2}}&=\left(n(x)-n(x_{1})\right)\left(1-\frac{f(x_{2},y_{2},v_{2})}{f(x,y_{2},v_{2})}\right)\\
&+\frac{a\varphi_{1}(y_{2})+p\varphi_{1}(y_{2})\varphi_{2}(z_{2})}{G_{1}}\left(1-\frac{f(x_{2},y_{2},v_{2})}{f(x,y_{2},v_{2})}+\frac{f(x,y,v)v}{f(x,y_{2},v_{2})v_{2}}\right)\\
&+\frac{a\varphi_{1}(y_{2})+p\varphi_{1}(y_{2})\varphi_{2}(z_{2})}{G_{1}}\left(1-\frac{1}{G_{1}}\int^{\infty}_{0}f_{1}(\tau)e^{-\alpha_{1}\tau}\frac{\varphi_{1}(y_{2})}{\varphi_{1}(y)}\frac{v(t-\tau)}{v_{2}}\frac{f(x(t-\tau,y(t-\tau),v(t-\tau)))}{f(x_{1},y_{2},v_{2})}\right)d\tau\\
&+\frac{a\varphi_{1}(y_{2})+p\varphi_{1}(y_{2})\varphi_{2}(z_{2})}{G_{1}}\left[\frac{1}{G_{1}}\int^{\infty}_{0}f_{1}(\tau)e^{-\alpha_{1}\tau}\textnormal{ln}\left(\frac{f(x(t-\tau),y(t-\tau),v(t-\tau))v(t-\tau)}{f(x_{2},y_{2},v_{2})v_{2}}\right)d\tau \right]\\
&+\frac{a\varphi_{1}(y_{2})+p\varphi_{1}(y_{2})\varphi_{2}(z_{2})}{G_{1}}\left[\frac{1}{G_{2}}\int^{\infty}_{0}f_{2}(\tau)e^{-\alpha_{2}\tau}\textnormal{ln}\left(\frac{\varphi_{1}(y(t-\tau))}{\varphi_{1}(y_{2})}d\tau\right)-\textnormal{ln}\left(\frac{f(x,y,v)v\varphi_{1}(y)}{f(x_{2},y_{2},v_{2})v_{2}\varphi_{1}(y_{2})}\right)\right].
\end{align*}

Therefore

\begin{align*}
\dot{V_{2}}&=\left(n(x)-n(x_{1})\right)\left(1-\frac{f(x_{2},y_{2},v_{2})}{f(x,y_{2},v_{2})}\right)\\
&+\frac{a\varphi_{1}(y_{2})+p\varphi_{1}(y_{2})\varphi_{2}(z_{2})}{G_{1}}\left(1-\frac{v}{v_{2}}+\frac{f(x,y_{2},v_{2})}{f(x,y,v)}+\frac{f(x,y,v)v}{f(x,y_{2},v_{2})v_{2}}\right)\\
&-\frac{a\varphi_{1}(y_{2})+p\varphi_{1}(y_{2})\varphi_{2}(z_{2})}{G_{1}}\left[H\left(\frac{f(x_{2},y_{2},v_{2})}{f(x,y_{2},v_{2})}\right)+H\left(\frac{f(x,y_{2},v_{2})}{f(x,y,v)}\right)\right]\\
&-\frac{a\varphi_{1}(y_{2})+p\varphi_{1}(y_{2})\varphi_{2}(z_{2})}{G_{1}}\left[\frac{1}{G_{1}}\int^{\infty}_{0}f_{1}(\tau)e^{-\alpha_{1}\tau}H\left(\frac{\varphi_{1}(y_{2})}{\varphi_{1}(y)}\frac{v(t-\tau)}{v_{2}}\frac{f(x(t-\tau),y(t-\tau),z(t-\tau))}{f(x_{2},y_{2},v_{2})}\right)d\tau \right]\\
&-\frac{a\varphi_{1}(y_{2})+p\varphi_{1}(y_{2})\varphi_{2}(z_{2})}{G_{1}}\left[\frac{1}{G_{2}}\int^{\infty}_{0}f_{2}(\tau)e^{-\alpha_{2}\tau}H\left(\frac{\varphi_{1}(y(t-\tau))v_{2}}{\varphi_{1}(y_{2})v}\right)d\tau\right].
\end{align*}

Therefore $\dot{V_{2}}\leq 0$, thus $E_{2}$ is stable. $\dot{V_{2}}=0$ if and only if
$x=x_{2}, y=y_{2}, v=v_{2}, z=z_{3}$. So, the largest compact invariant set in $\{(x,y,v,z):\dot{V_{2}}=0\}$ is
the singleton $E_{2}$. From LaSalle invariance principle, we conclude that $E_{2}$ is globally
asymptotically stable.
\end{proof}

\section{Numerical simulations }
In this section we present some numerical simulations to illustrate the results of stability obtained in our theorems of previous sections.
\begin{example}
Consider the functions  $n(x)= \lambda-dx+ rx \left( 1- \frac{x}{K} \right)$, $\phi_1(y)=y, \phi_2(z)=z, w(y,z)=yz$ and $f(x,y,v)= \frac{\beta x}{\alpha y+ \gamma x}$. Let $\tau_1, \tau_2 \in [0, \infty)$ two fixed delays, and set  $f_1(\tau)= \delta(\tau- \tau_1)$, $f_2(\tau)= \delta(\tau- \tau_2)$, $f_3(\tau)= \delta(\tau)$, where $\delta$ is the Dirac delta function defined as 

$$ \int_0^{ \infty} \delta(\tau- \tau_i) F(\tau)d \tau= F(\tau_i). $$

Then, the model takes the form:

\begin{align*}
\dot{x} &= \lambda-dx+ rx \left( 1- \frac{x}{K} \right)- \frac{\beta x v}{\alpha y+ \gamma x}, \\
\dot{y} &= \frac{\beta x(t- \tau_1) v(t- \tau_1)}{\alpha y(t- \tau_1)+ \gamma x(t- \tau_1)} e^{- \alpha_1 \tau_1}- a y - pyz, \\
\dot{v} &= k e^{- \alpha_2 \tau_2} y(t- \tau_2)-uv, \\
\dot{z}& = cyz-bz,
\end{align*}

 Fix the parameters as $\lambda=200, d=0.1, r=0.6, K=500,  p=1, k=0.8, \alpha_2=0.05, u=3.5, c=0.03, b=0.75, \tau_1=5, \tau_2=10, \alpha_1=0.1, \alpha= \gamma=0.001$. The trivial equilibrium point is given by $E_0=(666.6666,0,0,0)$, so the basic reproduction number $\mathcal{R}_0$ is obtained by:

 $$R_0= \frac{kG_1 G_2f(x_0,0,0)}{au}= \frac{k e^{-\mu \tau_1} e^{-\alpha_2 \tau_2} \beta }{au \gamma}= 105.10841176326923474 \beta.$$

 $\mathcal{R}_0 \leq 1$ iff $\beta \leq 0.009513986$. We set $\beta=0.003$ . By theorem \eqref{theo2} i), we have the single equilibrium $E_0=(666.6666,0,0,0)$ which is globally asymptotically stable by theorem \eqref{theo3}.

 Using a constant history function $S=(25,50,10,5)$ for $t \in (0,10)$ and DDE 23 tool from Matlab,  we obtain the figure \eqref{fig1}. \par  
\end{example}

\begin{example}
  Now, set $\beta=0.0096$, so $R_0>1$. Computing $R_1$ from its definition we have $R_1= \frac{ke^{-\alpha_1 \tau_1} e^{-\alpha_2 \tau_2}}{au} f(\hat{x}, \hat{y}, \hat{v})$, where

  $$ \hat{y}= \frac{b}{c}, \quad \hat{v} = \frac{\hat{y}k e^{-\alpha_2 \tau_2}}{u}, \quad n(\hat{x})- f(\hat{x}, \hat{y}, \hat{v})\hat{v}=0, $$

therefore $R_1=0.97091<1$ and we have then, a second equilibrium 

$$E_1=(659.461141,5.962025,0.826548,0).$$

 Using the history function from previous example we can plot the solution with Matlab. By theorem \eqref{theo4}, $E_1$ is globally asymptotically stable as we can see in figure \eqref{fig2}. \par 

   Finally, for $\beta=1$, we have $R_0>1$ and $R_1= 6.088529$, so there exists an infection free equilibrium $E_0$ and two equilibria 

   $$E_1=(1.461792, 152.184859,21.098236,0),$$  and $$E_2= (1.537198,25,3.465889,4.070823).$$
  
    By theorem \eqref{theo5} we have global stability of equilibrium $E_2$. The solutions are showed in figure \eqref{fig3}. 
\end{example}

  \begin{example}
In order to show that our model, generalizes the previous articles, we propose the following incidence function to show our results:

$$f(x,y,v)= \frac{\beta x }{(1+ \alpha y)(1+ \gamma v)},$$

this function satisfies
\begin{itemize}
 \item[i)] $f(0,y,v)=0, \quad \forall y,v \geq 0$.
 \item [ii)] $ \frac{\partial f}{ \partial x }=  \frac{\beta }{(1+ \alpha y)(1+ \gamma v)} >0, \quad \forall x,y,v >0  $.
 \item [iii a)] $ \frac{\partial f}{ \partial y }= -{\frac {\beta\,x \alpha }{ \left( \alpha y+1 \right) ^{2} \left( \gamma v+1 \right) }} \leq 0  \quad\forall x,y,v \geq 0 $.
 \item [iii b)] $ \frac{\partial f}{\partial v} = -{\frac {\beta\,x \gamma}{ \left( \alpha y+1 \right)  \left( \gamma v+1 \right) ^{2}}}\leq 0, \quad \forall x,y,v \geq 0. $ 
 \end{itemize}
 
Set the other functions and parameters as in example 1, then we obtain model:

\begin{align*}
\dot{x} &= \lambda-dx+ rx \left( 1- \frac{x}{K} \right)- \frac{\beta xv }{(1+  y)(1+ v)}, \\
\dot{y} &= \frac{\beta x(t- \tau_1)v(t- \tau_1) }{(1+  y(t- \tau_1))(1+v(t- \tau_1))} e^{- \alpha_1 \tau_1}- a y - pyz, \\
\dot{v} &= k e^{- \alpha_2 \tau_2} y(t- \tau_2)-uv, \\
\dot{z}& = cyz-bz,
\end{align*}

The infection free equilibrium is $E_0=(666.6666,0,0,0)$ as in previous cases, with $R_0=70.0722745 \beta$, so $R_0>1$ iff $\beta>0.01427097$. Therefore, if we set $\beta= 0.1$ then we have $R_0>1$, $R_1= 4.9234560677357286281$ and there exists two more equilibria points, 

$$E_1= (115.436331,183.268932,25.407594,0), \quad E_2= (481.791432,25,3.465889,3.138764)$$

 with $E_2$  globally asymptotically stable. The solutions are showed in figure \eqref{fig4}.
\end{example}

\section{Conclusions}
In this paper we studied the global properties of a model of infinitely distributed delayed viral infection, that considers a nonlinear CTL immune response, given by $w(y,z)=\phi_1(y) \phi_2(z)$ and a general incidence function of the form $f(x,y,v)v$, where $w$ and $f$ satisfy certain conditions derived from previous works and biological meanings. Even when there exists variety of papers that include the CTL immune response (see for example \cite{nowak1996population, li2015global, yang2015stability, yang2015analyzing, hattaf2012global}) and general incidence functions of various types (see \cite{yang2015analyzing, hattaf2012global, shu2013global} ), the model proposed in this article includes a family of the   works studied by several authors, and their conclusions can be seen as a particular case of our theorems. There lies its importance and relevance. \par 

The model presents always  an infection free positive equilibrium $E_0=(\bar{x},0,0,0)$, and two types of chronic infection equilibria: the CTL inactivated infection equilibrium (CTL-IE) $E_1=(x_1,y_1,v_1,0)$ and the CTL activated infection equilibrium (CTL-AE) $E_2=(x_2,y_2,v_2,z_2)$. The coexistence of these equilibria is determined by the basic reproduction number $R_0$ and the viral reproduction number $R_1$, those were defined in section 2 and are given in terms of parameters and the functions $f(x,y,v), f_i(\tau), \phi_1(y)$ and $\phi_2(z)$. The results show that $R_0>R_1$ and  the system  admits always a positive infection free equilibrium $E_0$, which is the unique equilibrium when $R_0\leq 1$. If $R_0>1$ then, in addition to $E_0$ we have only the CTL-IE (when $R_1\leq 1$), or the coexistence of the CTL-IE and CTL-AE ($R_1>1$). \par 

We proved, by construction of a Lyapunov function, that whenever the equilibrium $E_0$ is unique ($R_0 \leq 1$) and $R_0 \neq 1$, $E_0$ is globally asymptotically stable, moreover when $R_0>1$ and $R_1<1$ the CTL-IE, $E_1$ is  globally asymptotically stable. In the case of CTL-AE, $E_2$ we obtained conditions for global stability only in the case $f_3(\tau)= \delta (\tau)$, \textit{ie}, when the equation of $\dot{z}$ does not present delay. The results indicate that in this case, the equilibrium $E_2$ is globally asymptotically stable when $R_1>1$ and conditions $i)-H_4$ hold. It will be of interest to find conditions
that guarantee the global stability of the $E_2$  with a general $f_3(\tau)$, this topic can be taken as a future work.

\section*{Acknowledgments}
This article was supported in part by Mexican SNI under grant 15284 and CONACYT Scholarship 295308.

\section{Figures}

\begin{figure}[H]
\includegraphics[scale=0.5]{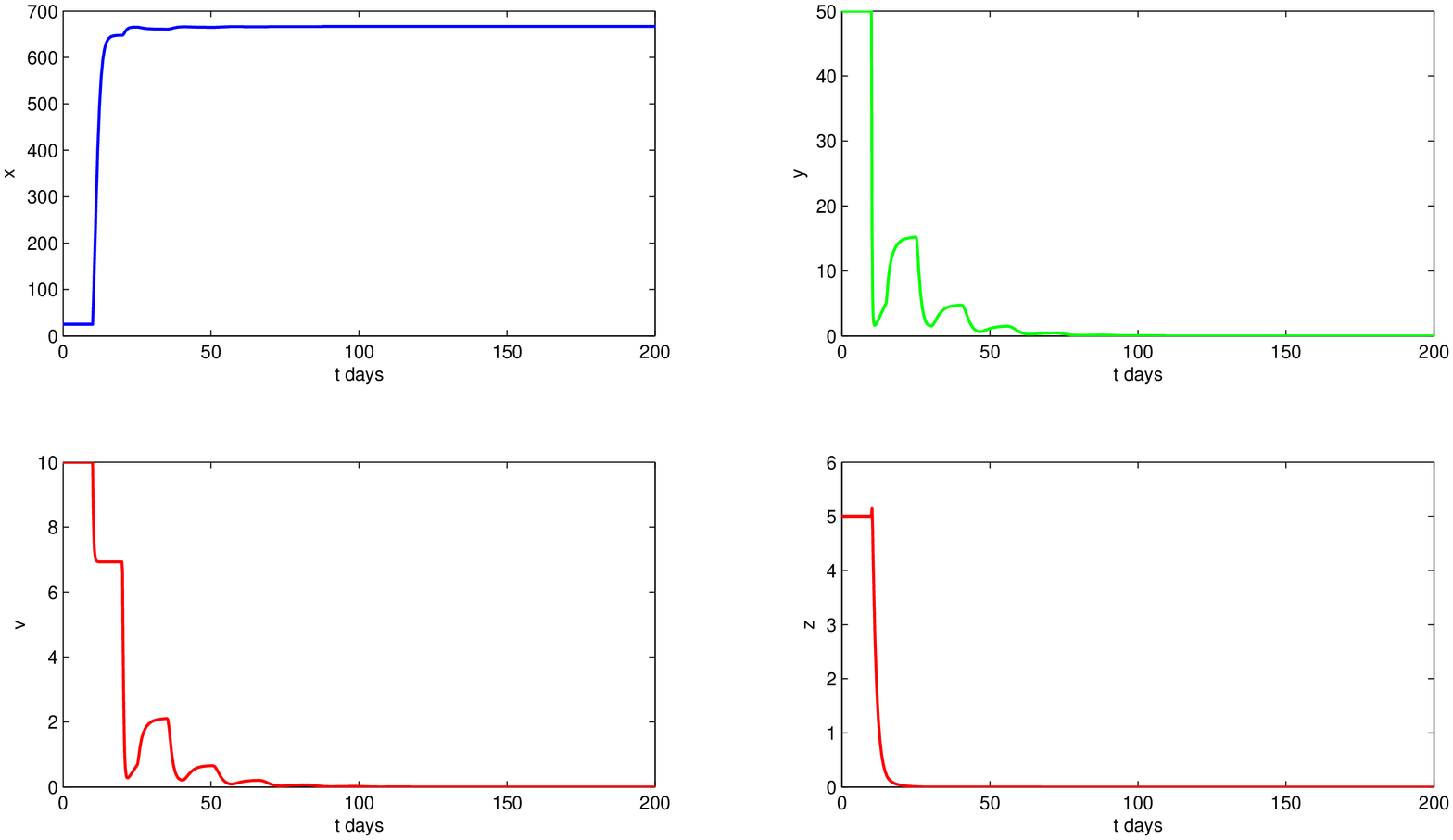} 
\caption{Global stability of infection free equilibrium for $\mathcal{R}_0<1$.} \label{fig1}
\end{figure}

\begin{figure}[H]
\begin{center}
\includegraphics[scale=0.5]{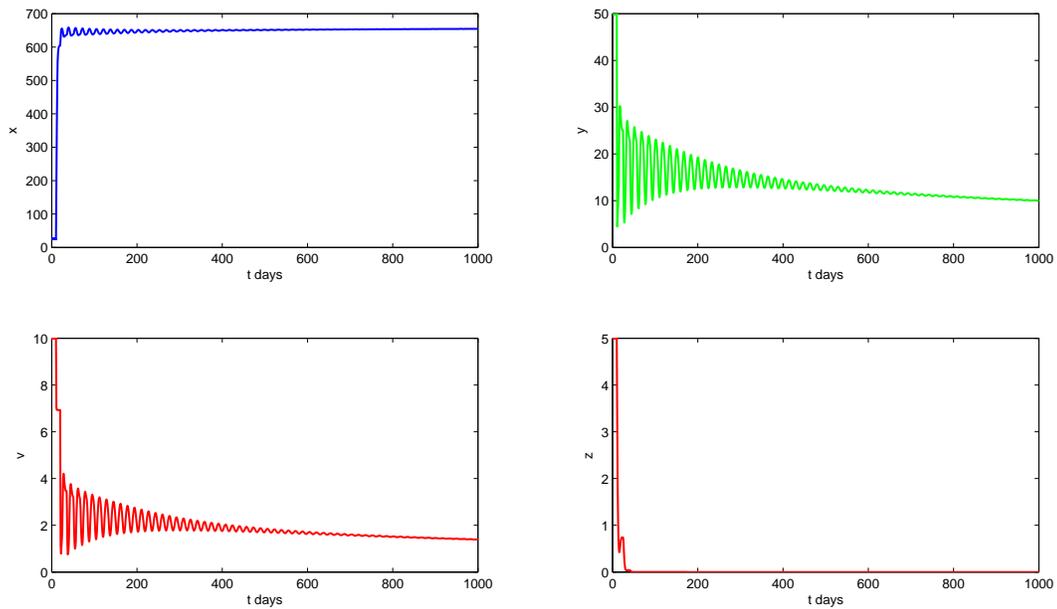} 
\end{center}
\caption{Global stability of CTL-IE $E_1$ for $\mathcal{R}_0>1>R_1$.} \label{fig2}
\end{figure}

\begin{figure}[H]
 \includegraphics[scale=0.5]{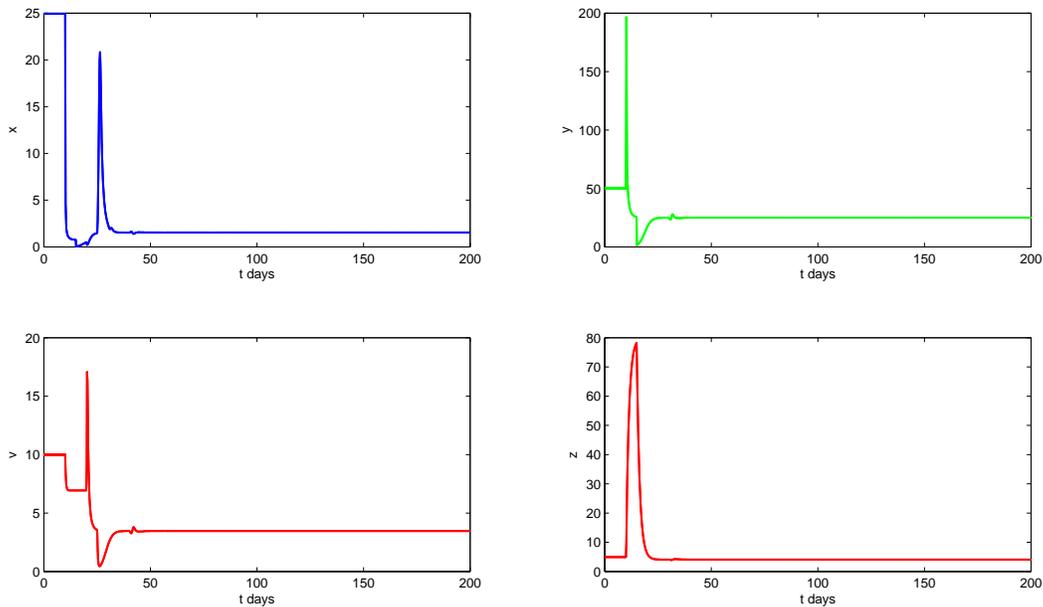}
 \caption{Global stability of CTL-AE $E_2$ for $R_0>1, R_1>1$.} \label{fig3}
\end{figure} 

\begin{figure}[H]
 \includegraphics[scale=0.5]{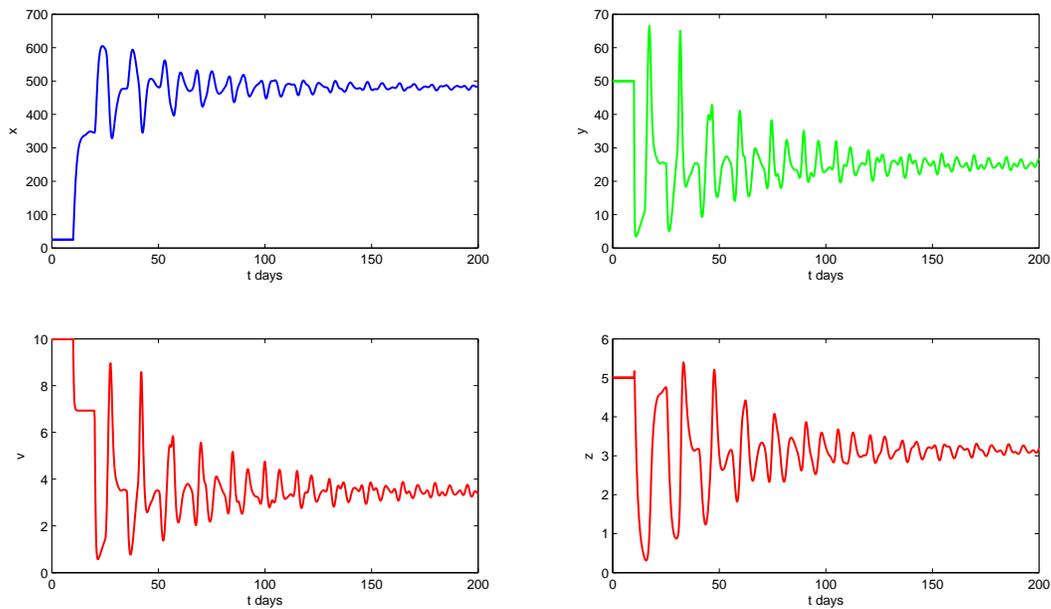}
 \caption{Global stability of CTL-AE $E_2$ for $f(x,y,v)= \frac{\beta x}{(1+\alpha y)(1+ \gamma v)}$, with $R_0>1, R_1>1$.} \label{fig4}
 \end{figure}

\end{document}